\documentclass[a4paper,11pt]{article}
\usepackage[utf8]{inputenc}
\usepackage{amsmath}
\usepackage{amsthm}
\usepackage{amsfonts}
\usepackage{amssymb}
\usepackage{graphicx}
\usepackage{color}
\usepackage{enumerate}
\usepackage{mathrsfs}
\usepackage{epsfig, url}
\usepackage{textcomp}
\usepackage{paralist}
\usepackage{authblk}
\usepackage{a4wide}

\newcommand{\disp}{\displaystyle}

\numberwithin{equation}{section}

\newtheorem{theorem}{Theorem}
\newtheorem{lemma}[theorem]{Lemma}
\newtheorem{proposition}[theorem]{Proposition}
\newtheorem{corollary}[theorem]{Corollary}

\newtheorem{definition}[theorem]{Definition}

\theoremstyle{remark}
\newtheorem{remark}[theorem]{Remark}

\numberwithin{theorem}{section}
\newcommand{\IND}{\mathbf{1}}

\newcommand{\N}{\mathbb{N}}

\newcommand{\bes}{\begin{equation*}}
\newcommand{\ees}{\end{equation*}}
\newcommand{\beas}{\begin{eqnarray*}}
\newcommand{\eeas}{\end{eqnarray*}}
\newcommand{\bea}{\begin{eqnarray}}
\newcommand{\eea}{\end{eqnarray}}
\newcommand{\be}{\begin{equation}}
\newcommand{\ee}{\end{equation}}
\newcommand{\bbl}{\begin{block}}
\newcommand{\ebl}{\end{block}}

\newcommand{\ve}{\varepsilon}

\newcommand{\cX}{\Gamma}
\newcommand{\cXt}{\tilde \cX}

\newcommand{\facn}{\lfloor n \rfloor}
\newcommand{\fac}{\lfloor 2 \rfloor}

\newcommand{\BOP}{\mathfrak{S}}
\newcommand{\ADP}{\mathfrak{A}}
\newcommand{\ADPa}{\mathfrak{A}^c}

\newcommand{\CS}{\mathcal{M}^\cdot(\cX)}
\newcommand{\CSt}{\mathcal{M}^\cdot(\tilde\cX)}

\newcommand{\CSf}{\mathcal{M}^\cdot_f(\cX)}

\newcommand{\Prob}{\mathcal{P}}
\newcommand{\Camp}{\mathcal{C}}

\newcommand{\R}{\mathbb{R}}

\newcommand{\B}{\mathfrak{M}_1}
\newcommand{\cc}{\mathfrak{a}}

\newcommand{\PerOU}{P_{\nu,c}^{OU}}
\newcommand{\poi}{\mathfrak{poi}}
\newcommand{\probap}{\mathfrak{p}}
\newcommand{\probaq}{\mathfrak{q}}
\newcommand{\Prho}{\mathbf{P}_{\hspace{-0.8mm}\rho}}
\newcommand{\PP}{\mathbf{P}}
\newcommand{\gbeta}{g_{\hspace{-0.3mm}\beta}}
\newcommand{\Pnu}{{\mathbb{P}_\nu}}
\newcommand{\Pnuxy}{{\mathbb{P}_\nu^{x,y}}}
\newcommand{\Pnuxx}{{\mathbb{P}_\nu^{x,x}}}
\newcommand{\Pnun}{{\mathbb{P}_{\nu^n}}}
\newcommand{\Enu}{{\mathbb{E}_\nu}}

\newcommand{\Prhoc}{\mathbf{P}^\mathfrak{a}_{\hspace{-0.8mm}\rho}}
\newcommand{\EPrhoc}{\mathbf{E}_{\mathbf{P}^\mathfrak{a}_{\hspace{-0.8mm}\rho}}}
\newcommand{\EPrhoo}{\mathbf{E}_{\mathbf{P}^0_{\hspace{-0.8mm}\rho}}}
\newcommand{\K}{\mathcal{K}}

\date{January 23, 2018}
\title{Conditioned point processes \\ with application to L\'evy bridges}
\author[1]{Giovanni Conforti\thanks{giovanni.conforti@polytechnique.edu}}
\author[2]{Tetiana Kosenkova\thanks{kosenkova@math.uni-potsdam.de}}
\author[3]{Sylvie Roelly\thanks{roelly@math.uni-potsdam.de}}
\affil[1]{\small D\'{e}partement de math\'{e}matiques appliqu\'{e}es, \'Ecole Polytechnique, Universit\'e Paris-Saclay. Route de Saclay, 91128 Palaiseau, France}
\affil[2,3]{\small Institute for Mathematics, University of Potsdam, Karl-Liebknecht-Strasse 24-25, 14476 Potsdam, Germany}

\begin{document}
\maketitle
\begin{abstract}
  Our first result concerns a characterisation by means of  a functional equation of Poisson point processes conditioned by the value of their first moment. It leads to a generalised version of Mecke's formula. {\it En passant}, it also allows to gain quantitative results about stochastic domination for Poisson point processes under linear constraints. \\
Since bridges of a pure jump L\'evy process in $\R^d$ with a height $\cc$ can be interpreted as a Poisson point process on space-time conditioned by pinning its first moment to $\cc$, our approach allows us to characterize bridges of L\'evy processes by means of a functional equation.
The latter result has two direct applications:
first we obtain a constructive and simple way to sample L\'evy bridge dynamics; second it allows to estimate the number of jumps for such bridges. We finally show that our method remains valid for linearly perturbed L\'evy processes like periodic Ornstein-Uhlenbeck processes driven by L\'evy noise.
\end{abstract}
\vspace{0.5cm}


\section{Introduction and notations}
In this paper we first consider Poisson point processes conditioned to satisfy linear constraints.
 As we will see later, they arise quite naturally in various situations, when studying bridges of L\'evy processes or periodic Ornstein-Uhlenbeck processes. What makes their study mathematically interesting (and intricate) is the fact that, in contrast with the Gaussian case where linear conditionings preserve Gaussianity, linear conditionings of Poisson point processes are no longer Poissonian. We propose a characterization of these conditional laws in Theorem \ref{th:main}
through the functional equation (*) which can be seen as a generalized version of the celebrated iterated Mecke identity. \\
Recall that Mecke's formula 
quantifies how much adding or removing a point from a random point configuration affects its probability. In our formula, indeed, we balance the cancellation and addition of points in such a way that the constraint is preserved. Let us precise our approach. Consider a Poisson point process $\PP(d\mu)$ on $\R$ under the linear constraint that the first moment $\B(\mu):= \int_\R x \mu(dx)$ of any point configuration $\mu$ is fixed to be equal to $\cc$.
To analyse the conditioned probability $\PP(d\mu \,|\B=\cc)$ we introduce an integro-difference operator on point measures $\mu$, which cancels a randomly chosen point $x$ of the support of $\mu$ and create two new points at places $x'$ and $x''$ whose sum $x' + x''$ equals that of the removed one, $x$. Therefore the first moment of the transformed point measure remains unchanged, equal to $\B(\mu)$. \\
Identity (*) will also be used to dominate stochastically with a Poisson random variable the law of the total mass of a Poisson point process conditioned by its first moment. This result is of particular interest since these conditioned laws cannot be computed in explicit form. Our result furnishes upper- or lower-bounds. \\

The main purpose of our study is presented in Section \ref{sec:3}. Considering a pure jump process as a point measure on a space-time set, we transpose our previous results in order to obtain a characterization of  bridges of pure jump L\'evy processes as the unique solutions of a functional equation. Indeed the former constraint on the first moment $\B$ corresponds in this context to fix the global size of the jumps of a path, or equivalently the height of the bridge.\\
Notice that in Equation \eqref{e7} which characterizes the set of pure jump processes having the same bridges than a given pure jump L\'evy processes $\Pnu$, a parameter $\chi_\nu$ appears,  called reciprocal characteristic. This bivariate function is computed from the diffuse jump measure $\nu$ of $\Pnu$ and encodes all the necessary information to construct the bridges.
In this respect, our result extends to the case of diffuse jumps the study of bridges of jump process  which so far was limited to discrete jump measures or random walks on graphs, see \cite{CDPR,CR,CL15}. Furthermore, following a first quantization strategy outlined in \cite{hairer2005analysis,hairer2007analysis}, our characterization can be used to construct a dynamics whose invariant law is a L\'evy bridge, see Subsection \ref{subs:simulation}.\\

The paper is organized as follows.
In Section 2 we exhibit in Theorem \ref{th:main} a characterization formula for Poisson point processes conditioned by their first moment. In particular, we deduce from that explicit stochastic comparisons results.
In Section 3 we apply our former characterization to bridges of pure jump L\'evy processes, whereas in Section 4
we apply them to the study of periodic Ornstein-Uhlenbeck processes driven by a L\'evy process.\\

Let us now introduce some useful notations which will appear in the paper.
\begin{itemize}
\item On a measured state space $\cX$ we consider $\CS$ (resp. $\CSf$), the set of point measures (resp. finite point measures) over $\cX$. \\
If a point $\gamma \in \cX$ belongs to the set of atoms of $\mu \in \CS$ we simply write $\gamma \in \mu$. Therefore, if $\mu $ is not reduced to the zero measure, denoted by $\underline{0}$,  $\mu = \sum_{\gamma\in \mu} \delta_\gamma$.
\item $\Prob (X)$ is the set of probability measures on a space $X$.
In particular
$$\poi_\lambda \in \Prob(\N)$$
 is the law of a
Poisson random variable with mean $\lambda >0$ and
$$
\Prho \in \Prob(\CS)
$$
 denotes
 the Poisson point process on  $\cX$ of intensity $\rho(d\gamma)$, where $\rho$ is a positive 
finite measure on $\cX$.
\item For any point measure $\mu \in \CS $,  its $n^{th}$ factorial product, $n \in \N^*$, is defined as the point measure on the product space $\cX^{\otimes n}$ given by
\be\label{e0}
\mu^{\lfloor n \rfloor} (d\gamma_1,\cdots, d\gamma_n) := \mu(d\gamma_1) (\mu- \delta_{\gamma_1})(d\gamma_2)\cdots
(\mu- \cdots -\delta_{\gamma_{n-1}})(d\gamma_n)
\ee
In other words (see e.g. \cite[p. 70]{daley2007introduction} )
$$
(\gamma_1, \cdots, \gamma_n) \in \mu^{\facn} \Leftrightarrow \forall i, \gamma_i \in \mu \textrm{ and } \not \exists i,j \textrm{  such that }\  \gamma_i = \gamma_j .
$$
In particular, since the point measure $\mu^{\fac}$ on  $\cX^{\otimes 2}$ satisfies
\be\label{e1}
\mu^{\fac}:= \sum_{\gamma,\gamma' \in \mu} \delta_{(\gamma,\gamma')} - \sum_{\gamma \in \mu} \delta_{(\gamma,\gamma)} ,
\ee
its support is the product of the support of $\mu$ with itself minus the diagonal.
\end{itemize}

\section{Splitting and conditioning a Poisson point process}

\subsection{
Mecke bivariate formula as tool to characterize a Poisson point process}

Let us shortly recall in this subsection how useful (reduced) Campbell measures are to characterize a Poisson point process.

First define on the product space $\Gamma \times \CS$ the map $\varsigma_+$ which adds an atom to a point  measure:
\be \label{def:addapoint}
\forall (\gamma,\mu)  \in \Gamma \times \CS, \quad \varsigma_+(\gamma,\mu) := (\gamma,\mu + \delta_\gamma)
\ee
Its inverse map $\varsigma^-$ is only defined on the set  $\{(\gamma,\mu) : \gamma \in \mu\} \subset \Gamma \times \CS$. It cancels one atom of a point measure:
\be \label{def:cancelapoint}
\varsigma_-(\gamma,\mu) := (\gamma,\mu - \delta_\gamma) .
\ee
Let us also introduce the bivariate version of $\varsigma_+$ corresponding to the addition of two atoms to a  point measure:
\be \label{def:add2points}
\forall (\gamma,\gamma',\mu)  \in \Gamma^2 \times \CS, \quad \varsigma_+^{(2)}(\gamma,\gamma',\mu) := (\gamma,\gamma',\mu + \delta_\gamma + \delta_{\gamma'}).
\ee
On the other side, the cancellation of two atoms of a point measure $\mu$ is defined and denoted as follows:
\be \label{def:cancel2points}
\textrm{ for }\gamma,\gamma' \in \mu, \gamma \not = \gamma' , \quad \varsigma^{(2)}_-(\gamma,\gamma',\mu) := (\gamma,\gamma',\mu - \delta_\gamma - \delta_{\gamma'}).
\ee

\begin{definition}[First order Campbell measures]
For any point process $Q$ on $\cX$, its one-to-one associated Campbell measure $\Camp_Q$ {\em (}resp. reduced Campbell measure $\Camp_Q^!${\em )} is defined as the following measure on $\Gamma \times \CS$:
$$
\Camp_Q (d\gamma,d\mu):= \mu(d\gamma) Q(d\mu) \quad \textrm{ resp.}
\quad \Camp_Q^! (d\gamma,d\mu):= \Camp_Q \circ (\varsigma_-)^{-1} (d\gamma,d\mu).
$$
\end{definition}

The celebrated Slivnjak-Mecke characterization offers an elegant identification of Poisson point processes via their Campbell measure. For any $\rho$, positive finite measure on $\cX$,
\be \label{eq:Mecke}
Q = \Prho \quad \Leftrightarrow \quad \Camp_Q^! = \rho \otimes Q \quad
\Leftrightarrow \quad \Camp_Q = \big(\rho \otimes Q \big) \circ (\varsigma_+)^{-1} .
\ee

\begin{remark}
For any $\gamma \in \cX$, denote by  $\Delta_\gamma=\delta_{\delta_\gamma}$ the degenerated (deterministic) point process concentrated on $\delta_\gamma \in \CS$. The latter identities \eqref{eq:Mecke} can be rewritten as
$$
\Camp_Q = \big(\rho(d\gamma) \Delta_\gamma (d\mu)\big) \star Q
$$
where $\star$ denotes the following generalized convolution between a measure $C$ on $\Gamma \times \CS$ and a point process $Q$ on $\Gamma$:
for any measurable positive test functions $F(\gamma,\mu)$ on $\Gamma \times \CS$,
$$
 \int F(\gamma,\mu) \, C \star Q \, (d\gamma,d\mu):= \int \int F(\gamma,\mu + \nu) \, C(d\gamma,d\nu)Q(d\mu).
$$
A generalisation of this equation in $Q$, where the measure $\rho(d\gamma) \Delta_\gamma (d\mu)$ is replaced by a more complicated one, is the subject of a recent study, see e.g. \cite{nehring2016splitting}.
\end{remark}
Iterating the above procedure, one can define Campbell measures with second (and higher) order, see e.g.
\cite[Eq. (15.6.14)]{daley2007introduction}.
\begin{definition}[Second-order Campbell measures]
For any point process $Q$ on $\cX$, one defines the second-order factorial Campbell measure $\Camp_Q^{(2)}$ on $\Gamma^2 \times \CS$ {\em (}resp. second-order reduced factorial Campbell measure $\Camp_Q^{(2),!}${\em )} as the following measure:
\begin{eqnarray*}
\Camp_Q^{(2)} (d\gamma,d\gamma', d\mu)&:=& \mu^{\fac}(d\gamma,d\gamma') Q(d\mu) \\
\textrm{ resp.}
\quad \Camp_Q^{(2),!} \,(d\gamma,d\gamma',d\mu)&:=& \Camp_Q^{(2)} \circ (\varsigma^{(2)}_-)^{-1}(d\gamma,d\gamma',d\mu),
\end{eqnarray*}
\end{definition}
Identities \eqref{eq:Mecke} then lead to the following Mecke's bivariate formula satisfied by the Poisson point process $\Prho$ (see e.g. \cite[p. 524]{daley2007introduction}) or \cite[Section 4.2]{lastpenrose}:
\begin{eqnarray} \label{eq:Meckebivariate}
\Camp_{\Prho}^{(2),!} \, (d\gamma,d\gamma',d\mu)&=& \rho \otimes \rho \otimes \Prho \ (d\gamma,d\gamma',d\mu) \nonumber\\
\Camp_{\Prho}^{(2)} (d\gamma,d\gamma',d\mu) &=& \big(\rho \otimes \rho \otimes \Prho\big) \circ (\varsigma_+^{(2)})^{-1} \, (d\gamma,d\gamma',d\mu) .
\end{eqnarray}

\subsection{A formula satisfied by the split Poisson point process} \label{sec:2.2}
From now on we need a group structure on the state space in order to define an addition and its inverse operation. For simplicity, we take for the rest of the paper
$\Gamma = \R^d $. We also suppose that the measure $\rho$ on $\gamma$ admits a density function with respect to Lebesgue measure denoted by $\rho$ too.\\
We then consider a {\it splitting transformation} on point measures on $\Gamma$ consisting in splitting  one of their atoms into two new ones,
in a specific way.
More precisely, define on the set  $\{(\gamma,\gamma',\mu) : \gamma \in \mu\} \subset \Gamma^2 \times \CS\setminus \{\underline{0}\}$
 the splitting map $\BOP$:
 \be \label{def:splitting}
(\gamma,\gamma',\mu) \mapsto
\BOP (\gamma,\gamma',\mu) := (\gamma -\gamma',\gamma',\mu - \delta_\gamma + \delta_{\gamma'} + \delta_{\gamma-\gamma'}).
\ee
The first order Campbell measure of a Poisson point process and  its second order Campbell measure are linked through the transformation $\BOP$ in the following way.
\begin{proposition}\label{prop:loopcharact}
Under the Poisson point process of intensity $\rho$, $\Prho$, the following identity holds:
\be\label{eq:1}
 \chi_{\rho} \ \Camp_{\Prho}^{(2)} =
 \ \big( \Camp_{\Prho} \otimes d\gamma' \big) \circ \BOP^{-1}
\ee
where the bivariate function $\chi_{\rho}$ satisfies
\be \label{eq:loopcharact}
\chi_{\rho}(\gamma,\gamma'):= \frac{\rho(\gamma+\gamma')}{\rho(\gamma) \rho( \gamma' ) }.
\ee
\end{proposition}
\begin{proof}
Integrate a positive test function $F$ under the left hand side of \eqref{eq:1}:
\begin{eqnarray*}
&&\int F(\gamma'', \gamma',\mu) \chi_{\rho}(\gamma'',\gamma') \ \Camp_{\Prho}^{(2)} (d\gamma'',d\gamma', d\mu) \\
&\stackrel{\eqref{eq:Meckebivariate}}{=}&
\int F(\gamma'', \gamma',\mu +\delta_{\gamma''} + \delta_{\gamma'} ) \chi_{\rho}(\gamma'',\gamma')\ \rho(d\gamma'')\rho(d\gamma') \Prho(d\mu) \\
&=&
\int F(\gamma -\gamma', \gamma',\mu + \delta_{\gamma-\gamma'}+\delta_{\gamma'} -\delta_{\gamma}+ \delta_{\gamma})
 \ \rho(d\gamma) \Prho(d\mu) d\gamma' \\
&\stackrel{\eqref{eq:Mecke}}{=}&
\int \big(F   \circ \BOP \big) \, (\gamma, \gamma',\mu) \
 \,\Camp_{\Prho}(d\gamma, d\mu)  d\gamma' ,
\end{eqnarray*}
which corresponds to the integral of $F$ under the right hand side of \eqref{eq:1}.
\end{proof}

\begin{corollary}\label{cor:densite}
Identity \eqref{eq:1} gains interesting interpretations by choosing the integrands in an appropriate way.
Since the intensity $\rho$ is  finite,  $\Prho$ a.s. carries
finite random point measures, that is $\mu(\Gamma)< + \infty $ a.s..
Now, take as test function $F$ a function of the following type:
$
\disp F(\gamma, \gamma',\mu):= \frac{\IND_{\mu(\Gamma)>1}}{\mu(\Gamma) - 1} \, \tilde F (\mu)
\varphi(\gamma'),$
where $\varphi$ is a probability density function.
Equality \eqref{eq:1} rewrites:
$$
E_{\Prho} \Big( \tilde F(\mu) D_\rho (\mu)
 \Big)
= E_{\Prho} \Big( \int_{\Gamma^2} \tilde F(\mu - \delta_\gamma + \delta_{\gamma'} + \delta_{\gamma-\gamma'} )
\frac{\mu(d\gamma)}{\mu(\Gamma)}  \varphi(\gamma')d\gamma' \Big)
$$
where
\be \label{def:densite}
 \disp D_\rho (\mu) := \frac{\int_{\Gamma^2} \chi_{\rho}(\gamma,\gamma') \varphi(\gamma') \ \mu^{\fac}(d\gamma,d\gamma')}{\mu(\Gamma) - 1} \IND_{\mu(\Gamma)>1}.
 \ee
This means that if you transform any realisation $\mu$ of the Poisson point process $\Prho$
as follows:
\begin{enumerate}[(1)]
\item if $\mu \not = \underline{0}$, select randomly one atom $\gamma$ of $\mu$
\item sample $\gamma' $ randomly according to the probability law with density $\varphi$
\item and replace the selected atom $\gamma$ by both atoms $\gamma'$ and $\gamma-\gamma'$;
\end{enumerate}
 then the obtained image measure is absolutely continuous with respect to $\Prho$ and the explicit density is expressed by \eqref{def:densite} in terms of the function  $\chi_\rho$.
\end{corollary}

\subsection{How to characterize the split Poisson point process pinned by its first moment}
Recall the definition of the {\em first moment } of a finite point measure $\mu\not =\underline{0}$ on $\Gamma$:
 $$
\B(\mu) := \int_{\Gamma} \gamma \,  \mu(d \gamma) =  \sum_{\gamma \in \mu} \gamma .
$$
Clearly one has $\B(\underline{0})=0$. \\
Remark that the first moment of a point measure, which is a random variable with values in $\Gamma$, remains invariant under the splitting transform $\BOP$ introduced
above:
\be \label{eq:firstmomentinvariant}
\B(\mu - \delta_\gamma + \delta_{\gamma'} + \delta_{\gamma-\gamma'}) = \B(\mu), \quad  \forall (\mu, \gamma, \gamma') \in \CSf \times \Gamma^2 .
\ee
The goal of this section is first, revisiting \eqref{eq:1}, to show that this identity remains true if one conditions the probability $\Prho$ by the event $\B^{-1}(\cc)= \{\mu: \B(\mu)= \cc \}$, $\cc \in \Gamma$; much more, we will prove that \eqref{eq:1} indeed characterizes the conditioned probability $\Prho^\cc(d\mu):=\Prho(d\mu \, | \B= \cc), \cc \not =0,$ within the set of probability measures on $\CSf$ with support included in  $\B^{-1}(\cc)$.\\
Notice that, since $\rho$ is diffuse, the law of $\B$ under $\Prho( \cdot \,|\{\underline{0}\}^c)$ is diffuse and therefore, for any $\cc \not =0$, the event $\{\B= \cc\}$ is $\Prho$-negligible. Nevertheless the conditioned probability $\Prho(d\mu \, | \B= \cc)$ can be constructed as limit measure for $\ve \rightarrow 0$ of the conditioned measures
$
 \Prho^{\cc,\ve} (\cdot):=\Prho( \cdot \ \vert \B \in B(\cc,\ve))
$
where $B(\cc,\ve)$ denotes the ball centered in $\cc$ with radius $\ve$.


\begin{theorem} \label{th:main}
Suppose $Q$ is a finite point process on $\cX$ and $\cc \in \Gamma\setminus \{0\}$. Then
\begin{displaymath}
   \left\{
   \begin{array}{rl}
     \chi_{\rho} \ \Camp_{Q}^{(2)}  \stackrel{(*)}{=} \big( \Camp_{Q} \otimes d\gamma' \big) \circ \BOP^{-1} & \\
     Q\big(\CSf \cap \B^{-1}(\cc) \big)  = 1&
   \end{array}
    \right.
    \Longleftrightarrow \quad   Q= \Prho^\cc.
\end{displaymath}
In other words,
$\Prho^\cc$ is the only finite point process on $\cX$ concentrated on the set $\{\B=\cc\}$
which fulfills the identity {\em (*)}.
\end{theorem}
{\em Proof } To prove  that $\Prhoc$ fulfills Identity (*) is straightforward.
Disintegrate the measure $\Prho$ along all possible values of $\B$:
$\Prho = \int \Prhoc\ \lambda_\rho (d\cc ) $ where $\lambda_\rho$ is the image measure of $\Prho$
under $\B$, and write the identity \eqref{eq:1} tested on functions defined on $\Gamma^2 \times  \CSf$ of the form $f(\B(\mu)) F(\gamma,\gamma',\mu)$. One obtains, using the invariance property \eqref{eq:firstmomentinvariant}:
\begin{eqnarray*}
&&\int f(\B(\mu)) \, F (\gamma,\gamma',\mu) \,\chi_{\rho}(\gamma,\gamma')\ \mu^{\fac}(d\gamma,d\gamma')
\Prhoc(d\mu) \lambda_\rho (d\cc)\\
&&=
\int f(\B\circ \BOP(\mu)) \,  F   \circ \BOP \,(\gamma,\gamma',\mu) \
 \mu (d\gamma) \Prhoc(d\mu)  d\gamma' \lambda_\rho (d\cc)\\
&\Longleftrightarrow &\\
&&\int f(\cc) \Big( \int F (\gamma,\gamma',\mu) \,\chi_{\rho}(\gamma,\gamma')\ \mu^{\fac}(d\gamma,d\gamma')
\Prhoc(d\mu) \Big) \lambda_\rho (d\cc)\\
&&=
\int f(\cc) \Big(\int F   \circ \BOP (\gamma,\gamma',\mu) \,  \
 \,\mu (d\gamma) \Prhoc(d\mu)  d\gamma' \Big) \lambda_\rho (d\cc).
\end{eqnarray*}
This is enough to deduce that (*) holds for $Q=\Prhoc$.\\
Before proving the implication from the left to the right in Theorem \ref{th:main}, we  develop some necessary tools.
First we introduce for any finite point process $Q$ its associated {\em diminished} point process  $Q^-$, which is constructed by removing one atom at random from any realization of $Q$:
\begin{definition}[Diminished point process]
The diminished point process $Q^-$ of a point process $Q \in \Prob(\CSf)$ which does not carry the zero measure is defined as follows:
for any positive test function $F$ on $\CSf$,
\be \label{eq:Q-}
E_{Q^-}(F) =
E_{Q}\left( \int_{\Gamma}  F (\mu - \delta_\gamma)\,  \frac{\mu(d\gamma)}{\mu(\Gamma)}  \right).
\ee
\end{definition}
%
%
The end of the (tricky) proof of Theorem \ref{th:main} is now a direct consequence of the next three propositions.\\
For $\cc\not = 0$ the conditioned point process $\Prhoc$, which does not carry the zero measure and is concentrated on a $\Prho$-negligible set, is singular with respect to $\Prho$. Nevertheless, it is remarkable that its diminished version $(\Prhoc)^-$ is absolutely continuous with respect to $\Prho$, as stated in the next proposition.
\begin{proposition} \label{prop:1}
 For any $\cc\not = 0$ the diminished conditioned Poisson point process $(\Prhoc)^-$ is absolutely continuous with respect to $\Prho$ and its density is proportional to $\disp \frac{ \rho(\cc-\B(\mu))}{\mu(\Gamma)+1}$.
\end{proposition}

\begin{proposition} \label{prop:2}
Suppose the finite point process $Q$ fulfills {\em (*)} and, for some $\cc \not =0$,  $Q(\B =\cc)=1$.
Then $Q^-$ is absolutely continuous with respect to $\Prho$ and its density is proportional to $\disp \frac{ \rho(\cc-\B(\mu))}{\mu(\Gamma)+1}$.
\end{proposition}
\begin{proposition} \label{prop:3}
Suppose the finite point process $Q$ is concentrated on \mbox{$\{\B = \cc\}$}. If its diminished version satisfies $Q^-= (\Prhoc)^-$ then  $Q= \Prhoc$.
\end{proposition}

\begin{proof} {\it of Proposition \ref{prop:1}}\\
Let us first prove that  $(\Prho^{\cc,\ve})^-$ is absolutely continuous w.r.t.  $\Prho$.
Take $\ve < |\cc|$. \\
Thus,  for any $\mu$ in the support of $\Prho^{\cc,\ve}$, $\B(\mu)$ does not vanish which implies that  $\Prho^{\cc,\ve}$ does not carry the zero measure.
Then for all functions $F$ bounded and measurable:
\beas
E_{(\Prho^{\cc,\ve})^-}(F) &=&
  \int \int_{\Gamma}  F (\mu - \delta_\gamma)\,  \frac{\mu (d\gamma)}{\mu(\Gamma)} \  \Prho^{\cc,\ve}(d\mu) \\
	&=& \frac{1}{Z^{\cc,\ve}_{\rho}} \int \int_{\Gamma}
\mathbf{1}_{B(\cc,\ve)}\circ \B\,(\mu) \,   F(\mu - \delta_\gamma) \,  \frac{1}{\mu(\Gamma)}\Camp_{\Prho}  (d\gamma,d\mu)  \\
		&=&\frac{1}{Z^{\cc,\ve}_{\rho}} \int \int_{\Gamma}
\mathbf{1}_{B(\cc - \gamma,\ve)} \circ \B \, (\mu) F (\mu) \,
\frac{d\mu}{\mu (\Gamma)+1} \Camp^!_{\Prho}  (d\gamma,d\mu)  \\
	&\stackrel{\eqref{eq:Mecke}}{=}& \frac{1}{Z^{\cc,\ve}_{\rho}} E_{\Prho} \left( F(\mu) \frac{1}{\mu(\Gamma)+1}\int_{\Gamma}\mathbf{1}_{B(\cc - \gamma,\ve)} \circ \B \, (\mu) \ \rho(d \gamma)   \right) \\
 &=&  \frac{1}{Z^{\cc,\ve}_{\rho}} E_{\Prho} \left( F(\mu) \frac{1}{\mu(\Gamma)+1}
\int_{B(\cc - \B(\mu),\ve)} \rho(\gamma) d \gamma \right)
\eeas
Therefore
$$
\frac{d(\Prho^{\cc,\ve})^- }{ d \Prho } (\mu) =  \frac{1}{ Z^{\cc,\ve}_{\rho}} \frac{1}{ \mu(\Gamma) +1} \int_{ B(\cc-\B(\mu),\ve) } \rho( \gamma) d \gamma
$$
where $ Z^{\cc,\ve}_{\rho}$ is the renormalising constant
 $ Z^{\cc,\ve}_{\rho}:=  \Prho( \B \in B(\cc,\ve)) $.\\
Now we pass to the limit as $\ve  \rightarrow 0$ and
check that $\frac{d(\Prho^{\cc,\ve})^- }{ d \Prho } $ converges. Clearly,\\
$
\int_{B(\cc - \B(\mu),\ve)} \rho(\gamma) d \gamma  =  \ O(\ve^d)
$
where $\rho$ is the density of the intensity measure of $\Prho$. On the other side, since the law of $\B$ under $\Prho(\cdot|\{\underline{0}\}^c)$ is absolutely continuous, $Z^{\cc,\ve}_{\rho} $ is also of order $\ve^d$ as $\ve \downarrow 0$.
This completes the proof of Proposition \ref{prop:1}.
\end{proof}
\begin{proof}  {\it of Proposition \ref{prop:2}}\\
Assume  that $Q$ satisfies (*). We have to show that
$$
\tilde{Q}  :=\frac{\mu(\Gamma)+1}{ \rho(\cc-\B(\mu))} \ Q^-.
$$
is indeed proportional to the Poisson process $\Prho$, or equivalently that $\tilde{Q} $ satisfies
Mecke's formula \eqref{eq:Mecke}.
Therefore we compute the integral of any  test function $F \in \mathcal{B}( \Gamma \times \CS)$
under the measure $ \big(\rho \otimes \tilde{Q} \big) \circ (\varsigma_+)^{-1} $:
\beas
 &&\int
F(\gamma',\mu +\delta_{\gamma'} )  \, \rho(d\gamma')\tilde Q(d\mu)\\
&=&
  \int_{\Gamma \times \CS}
F(\gamma',\mu +\delta_{\gamma'} )  \,  \frac{\mu(\Gamma)+ 1}{\rho( c-\B(\mu )) }
\, Q^-(d\mu) \rho(d\gamma') \\
&\stackrel{\eqref{eq:Q-}}{=}&
  \int_{\Gamma^2 \times \CS}
F(\gamma',\mu +\delta_{\gamma'}-\delta_\gamma )  \,
\frac{(\mu - \delta_\gamma)(\Gamma)+ 1}{\rho( \cc-\B(\mu-\delta_\gamma )) } \frac{\mu(d\gamma)}{\mu(\Gamma)}
\, Q(d\mu) \rho(d\gamma') \\
&=&
\int_{\Gamma^2 \times \CS}
F(\gamma',\mu +\delta_{\gamma'}-\delta_\gamma )  \,  \frac{\rho(\gamma')}{\rho(\gamma)}
\, \Camp_Q(d\gamma, d\mu) d\gamma' ,
\eeas
since $Q$ is concentrated on point measures with fixed first moment equal to $\cc$.
Now define the function
$\tilde{F} \in \mathcal{B}(\Gamma^2 \times \CS)$ by
\bes
\tilde{F}(\gamma'', \gamma',\mu) :=
\frac{\rho(\gamma')}{\rho(\gamma'' + \gamma')} F(\gamma',\mu - \delta_{\gamma''}) .
\ees
The above identity rewrites
\beas
 &&\int
F(\gamma',\mu +\delta_{\gamma'} )  \, \rho(d\gamma')\tilde Q(d\mu)\\
&=&
  \int_{\Gamma^2 \times \CS}
\tilde{F}\circ \BOP \,(\gamma, \gamma',\mu) \ \Camp_Q(d\gamma, d\mu) d\gamma'\\
&\stackrel{(*)}{=}&
  \int_{\Gamma^2 \times \CS}
\tilde{F} \,(\gamma, \gamma',\mu) \, \chi_\rho (\gamma, \gamma') \
\Camp_{Q}^{(2)}(d\gamma, d\gamma',d\mu)\\
&=& \int_{\Gamma^2 \times \CS}
F(\gamma',\mu -\delta_{\gamma} )  \, \frac{1}{\rho(\gamma)} \
\Camp_{Q}^{(2)}(d\gamma, d\gamma',d\mu)\\
&=& \int_{\Gamma^2 \times \CS}
F(\gamma',\mu -\delta_{\gamma} )  \, \frac{1}{\rho(\gamma)} \
\mu^{\fac}(d\gamma, d\gamma') Q(d\mu)\\
&=& \int_{\Gamma^2 \times \CS}
F(\gamma',\mu -\delta_{\gamma} )  \, \frac{\mu(\Gamma)}{\rho(\cc-\B(\mu-\delta_\gamma))}
(\mu-\delta_\gamma)(d\gamma') \frac{\mu(d\gamma)}{\mu(\Gamma)}
Q(d\mu)\\
&=& \int_{\Gamma \times \CS}
F(\gamma',\mu )  \, \frac{\mu(\Gamma)+1}{\rho(\cc-\B(\mu))}
\mu(d\gamma') Q^-(d\mu)\\
&=& \int_{\Gamma \times \CS}
F(\gamma',\mu )  \,
\Camp_{\tilde Q}(d\gamma', d\mu) .
\eeas
\end{proof}
\begin{proof}  {\it of Proposition \ref{prop:3}}\\
Due to the fact that $Q(\B=\cc) =1$, we can reconstruct $Q$ from $Q^-$, or equivalently,  $\Camp_Q^!$ from $\Camp_{Q^-}$:
\beas
\int F(\gamma,\mu)\Camp_Q^!(d\gamma,d\mu) &=& \int F(\gamma,\mu-\delta_\gamma)\, \mu(d\gamma)Q(d\mu) \\
&=& \int F(\gamma,\mu-\delta_\gamma) \big((\mu-\delta_\gamma)(\Gamma)+1 \big)
\frac{\mu(d\gamma)}{\mu(\Gamma)}Q(d\mu) .
\eeas
Now, $\gamma= \B(\delta_\gamma)= \B(\mu)- \B(\mu-\delta_\gamma)=\cc- \B(\mu-\delta_\gamma), Q$-a.s.. Therefore
\beas
\int F(\gamma,\mu)\Camp_Q^!(d\gamma,d\mu)
&=&  \int F(\cc- \B(\mu-\delta_\gamma),\mu-\delta_\gamma) \big((\mu-\delta_\gamma)(\Gamma)+1 \big)
\frac{\mu(d\gamma)}{\mu(\Gamma)}Q(d\mu) \\
 &=&  \int F(\cc- \B(\mu),\mu) \big(\mu(\Gamma)+1 \big)
Q^-(d\mu) \\
 &=&  \int \tilde F(\mu) \,
\Camp_{Q^-}(d\gamma,d\mu),
\eeas
where $\tilde F(\mu) := \disp \frac{\mu(\Gamma)+1}{\mu(\Gamma)} \,F(\cc- \B(\mu),\mu) $.
\end{proof}
\begin{remark}
Note that identity (*) is trivially satisfied by the degenerate point process $\delta_{\underline{0}}$   carrying only  the empty configuration. In that case left and right hand sides of (*) vanish. Moreover, since that identity  is linear as function of Q, any mixture of solutions of (*) remains a solution of (*).
This is the reason  why the atomic part  on $\underline{0}$ of a solution of (*) can not be quantified by (*) and why we have to consider separately the case $\cc=0$. \\
Therefore, if the support of $Q$ is included in $\{\B=0 \}$, developing the same arguments as above on its restriction to $\{\underline{0}\}^c$ leads to its characterization:
$$
 \chi_{\rho} \ \Camp_{Q}^{(2)}  = \big( \Camp_{Q} \otimes d\gamma' \big) \circ \BOP^{-1}
\quad     \Longleftrightarrow \quad   Q(\cdot \,|\{\underline{0}\}^c)= \Prho^0(\cdot \,|\{\underline{0}\}^c).
$$
\end{remark}
\subsection{Application: Stochastic comparison between the pinned Poisson point process and the unpinned one}

Our aim in this subsection is to apply Theorem \ref{th:main} to compare stochastically the density of the points of a pinned Poisson point process $\Prhoc$ with that of an unpinned Poisson point process $\Prho$, under specific assumptions on the intensity measure $\rho$.
We first recall the concept of dominance for probability laws on $\mathbb{N}$.
\begin{definition}
Let $\probap \in \Prob(\N)$ and $\probaq \in \Prob (\N)$ be two probability measures on $\mathbb{N}$. We say that $\probap$ dominates $\probaq$ (or equivalently $\probaq$ is dominated by $\probap$) if and only if the tails of $\probap$ are larger than the tails of $\probaq$ in the sense that, for any $j \geq 1$, we have
$$
\probap (\{n \in \N: n\geq j\}) \geq \probaq (\{n \in \N: n\geq j\}) .
$$
 In that case we denote
$\probap \succeq  \probaq$ $($or $\probaq \preceq \probap)$.
\end{definition}

\begin{proposition} \label{prop:comp}
\begin{enumerate}
\item
Assume that the density function $\rho$ satisfies on $\Gamma$:
\begin{equation}\label{eq:cconvup}
\exists K >0 \quad \forall \gamma \in \Gamma, \quad \rho \ast \rho \,(\gamma) \leq K  \ \rho(\gamma).
\end{equation}	
Then, for any $\cc \neq 0 $, the law of the number of points of the process
$\Prhoc$  is dominated by $\poi ^+_{2K}$, where $\poi ^+_{2K}\in \Prob(\N^*)$ denotes the Poisson law
conditioned to be positive:
\begin{equation}\label{eq:dominup}
\Prhoc (\mu(\Gamma)= \cdot  ) \preceq \poi ^+_{2K}(\cdot):= \frac{\poi_{2K} (\cdot)}{\poi_{2K }(\N^*)} .
\end{equation}
Moreover
\begin{equation}\label{eq:dominup0}
\Prho^0 (\mu(\Gamma)= \cdot \, | \{\underline{0}\}^c ) \preceq \poi ^+_{2K}(\cdot):= \frac{\poi_{2K} (\cdot)}{\poi_{2K }(\N^*)} .
\end{equation}
\item
If $\rho$ satisfies the converse condition
\begin{equation}\label{eq:cconvlow}
\exists\quad k >0 \quad \forall \gamma \in \Gamma, \quad \rho \ast \rho \,(\gamma) \geq k \ \rho(\gamma),
\end{equation}
then for any $\cc \neq 0$,
$$
\Prhoc (\mu(\Gamma)= \cdot  ) \succeq \poi ^+_{2k}(\cdot) .
$$
\end{enumerate}
\end{proposition}
\begin{proof}
We only prove the statement \eqref{eq:dominup}, the proof of \eqref{eq:cconvlow} being very similar.
Recall that, due to Theorem \ref{th:main}, for any positive test function $F$,
\begin{equation*}
E_{\Prhoc} \bigg( \int_{\Gamma^2}F(\gamma -\gamma',\gamma',\mu - \delta_{\gamma}+ \delta_{\gamma'}+\delta_{\gamma-\gamma'})\mu(d\gamma) d\gamma' \bigg)
= E_{\Prhoc}\bigg(  \int_{\Gamma^2} F(\gamma,\gamma',\mu) \chi_{\rho}(\gamma,\gamma') \mu^{\lfloor 2 \rfloor}(d \gamma ,d \gamma')\bigg)
\end{equation*}
By plugging in $G(\gamma,\gamma',\mu): = F(\gamma,\gamma',\mu) \chi_{\rho}(\gamma,\gamma')^{-1}$ we obtain
\begin{equation*}
\int \int_{\Gamma^2} \frac{G(\gamma -\gamma',\gamma',\mu - \delta_{\gamma}+ \delta_{\gamma'}+\delta_{\gamma-\gamma'})}{\chi_{\rho}(\gamma-\gamma',\gamma')} \mu(d\gamma)\, d\gamma'\,\Prhoc(d \mu)
=\int\int_{\Gamma^2} G(\gamma,\gamma',\mu)  \mu^{\lfloor 2 \rfloor}(d \gamma ,d \gamma') \Prhoc(d\mu)
\end{equation*}
If we consider functionals of the form $G(\gamma,\gamma',\mu) = g(\mu(\Gamma))$ for some measurable map $g : \N \rightarrow \mathbb{R}^+$, the right hand side of the equation above becomes
\[
\frac{1}{2} \int g(\mu(\Gamma)) \mu(\Gamma) (\mu(\Gamma)-1)  \Prhoc(d\mu) .
\]
We write the left hand side as
\begin{eqnarray*} \int g(\mu(\Gamma)+1 )
\int_{\Gamma}
\frac{1}{\rho(\gamma)} \Big(\int\rho(\gamma-\gamma')\rho(\gamma') d\gamma' \Big)\mu(d\gamma)
 =  \int g(\mu(\Gamma)+1 ) \int_{\Gamma} \frac{1}{\rho(\gamma)} \ \rho \ast \rho(\gamma) \,  \mu(d\gamma).      \end{eqnarray*}
Under assumption \eqref{eq:cconvup} the last term in the formula above is bounded by
\[
K \int g(\mu(\Gamma)+1) \mu(\Gamma) \Prhoc(d \mu)  .
\]
Therefore, we have proven that for any $g\geq 0$,
\[
\int g(\mu(\Gamma)) \mu(\Gamma) (\mu(\Gamma)-1)  \Prhoc (d\mu)  \leq 2  K \int g(\mu(\Gamma)+1) \mu(\Gamma) \Prhoc (d \mu),
\]
which is equivalent to say that for any $\bar g$ such that $\bar g(1)=0$,
\[
\int \bar g(\mu(\Gamma)) \mu(\Gamma)   \Prhoc(d\mu)  \leq 2  K \int \bar g(\mu(\Gamma)+1) \Prhoc (d \mu) .
\]
By choosing $\bar g=\mathbf{1}_{\{i\}}$, we obtain
\[
\Prhoc(\mu(\Gamma) =i) \, i \leq 2 K \, \Prhoc(\mu(\Gamma) =i-1), \quad \forall i\geq 2 .
\]
 Taking $j:=i-1$ and observing that
$$
\frac{2 K }{j+1}= \frac{\poi_{2K}(j+1)}{\poi_{2K}(j)} =
\frac{\poi ^+_{2K}(j+1)}{\poi ^+_{2K}(j)}, \quad j\geq 1,
 $$
the statement above can be rewritten as
\[
\forall j \geq 1, \quad \Prhoc(\mu(\Gamma)=j+1) \ \poi ^+_{2K} (j) \leq \Prhoc (\mu(\Gamma) =j)
\  \poi ^+_{2K} (j+1).
\]

Since $\cc \neq 0$, $ \Prhoc$ does not carry the zero measure and then $ \Prhoc(\underline{0})=0$. Thus, we can regard $\Prhoc(\mu(\Gamma) = \cdot )$ as a measure on $\mathbb{N}^*$. The desired conclusion now follows applying Lemma \ref{lm:domination} to  $\probap=\poi ^+_{2K}$ and \mbox{$\probaq := \Prhoc (\mu(\Gamma)=\cdot \, )$}.

\end{proof}
\begin{lemma}\label{lm:domination}
Let $\probap, \probaq$ be two probability laws on $\N^*$. Moreover, assume that $\probap$ is always positive. If
\[
\forall j \geq 1, \quad \probaq(j+1)\, \probap(j) \leq \probaq(j) \, \probap(j+1)
 \]
then $\probap \succeq  \probaq$.
\end{lemma}
\begin{proof}
Suppose first that both laws are positive, the general case following with a simple approximation argument. In that case we can rewrite the assumption as
\begin{equation}\label{eq:domination1}
\forall \, j \geq 1, \quad  \frac{\probaq(j+1)}{\probaq(j)} \leq \frac{\probap(j+1)}{\probap(j)} \quad
\Longrightarrow
\quad
\forall \,  k \geq j \geq 1 \quad \frac{\probaq(k)}{\probaq(j)} \leq \frac{\probap(k)}{\probap(j)}.
\end{equation}
We have to show that for all $j\geq 1$, $\sum_{k \geq j}\probap(k) \geq \sum_{k \geq j}\probaq(k)$. To do this it is sufficient to show that the function $g$ defined by
\[g:\N^* \rightarrow \mathbb{R}_{+}, \quad g(j) := \frac{\sum_{k \geq j }\probaq(k)}{ \sum_{k \geq j} \probap(k)} \]
is decreasing, and to remark that $g(1)=1$.
To show that $g$ is decreasing we observe that
\begin{eqnarray*}
g(j+1)-g(j) \leq 0 & \Leftrightarrow & \sum_{k \geq j+1}\probaq(k)\sum_{l \geq j}\probap(l)-\sum_{k \geq j+1}\probap(k)\sum_{l \geq j}\probaq(l) \leq 0 \\
&\Leftrightarrow & \probap(j)\sum_{k \geq j+1}\probaq(k)-\probaq(j)\sum_{k \geq j+1} \probap(k) \leq 0 \\
&\Leftrightarrow &\sum_{k \geq j+1}\frac{\probaq(k)}{\probaq(j)} \leq \sum_{k \geq j+1}\frac{\probap(k)}{\probap(j)}.
\end{eqnarray*}
This last condition is directly implied by \eqref{eq:domination1}.
\end{proof}
The following statement provides us the information about the expected number of points of the pinned process.
\begin{corollary}\label{cor:comp_exp}
1. Assuming that condition (\ref{eq:cconvup}) holds true and $\cc \neq 0$. Then
$$
\EPrhoc (\mu(\Gamma)= \cdot \ ) \leq \frac{2K}{1-e^{-2K}}
$$
Moreover
$$
\EPrhoo (\mu(\Gamma)= \cdot \ | \{\underline{0}\}^c) \leq \frac{2K}{1-e^{-2K}}
$$
2. Assuming that condition (\ref{eq:cconvlow}) holds true and $\cc \neq 0$. Then
$$
\EPrhoc (\mu(\Gamma)= \cdot \  ) \geq \frac{2k}{1-e^{-2k}} .
$$
Moreover
$$
\EPrhoo (\mu(\Gamma)= \cdot \ | \{\underline{0}\}^c) \geq \frac{2k}{1-e^{-2k}}
$$
\end{corollary}
\begin{proof}
The statement follows immediately from the fact that for a non-negative discrete-valued random variable $X$ the expectation rewrites as
$$
\mathrm{E}(X)=\sum_{j \geq 1} \mathrm{P}(X \geq j)
$$
and that $\poi_{2K }(n)=\frac{(2K)^{n}}{n!}\frac{e^{-2K}}{1-e^{-2K}}.$
\end{proof}

We will discuss in Section \ref{sec:compLevyBridges} several examples of measures $\rho$ satisfying  condition \eqref{eq:cconvup} and/or  condition \eqref{eq:cconvlow}.\\

One can generalize Proposition \ref{prop:comp} by comparing the random number of points inside of any cone of $\Gamma$ under the Poisson point process and its pinned version, as follows.

Fix a  cone $\K$ with positive Lebesgue measure. We define a convolution operation  $\stackrel{(\K)}{\ast}$ of a function $\rho$ with itself on the cone $\K$ as follows:
\begin{equation} \label{convolution_on_cone}
\rho \stackrel{(\K)}{\ast} \rho \,(\gamma) = \int_{\K\cap (\gamma - \K)} \rho(\gamma')\rho(\gamma-\gamma') \, d \gamma' , \quad \gamma \in \K .
\end{equation}
Let us remark that if $\gamma \in \K$, then the set $\K \cap (\gamma - \K)$ has positive Lebesgue measure as well, so that $\rho \stackrel{(\K)}{\ast} \rho \, (\gamma)>0$.
We can now express the following result.
\begin{proposition}
Suppose the density function $\rho$ satisfies:
\begin{equation}\label{eq:convup-cone}
\exists K>0 \quad \forall \gamma \in \K, \quad \rho \stackrel{(\K)}{\ast} \rho(\gamma) \leq K \ \rho(\gamma),
\end{equation}	
and let $\mu(\K)$ be the random number of points of $\mu$ in $\K$. Then, for any $\cc \neq 0$, the law of $\mu(\K)$ under $\Prhoc$ is dominated by $\poi ^+_{2K}$.\\
Conversely, if
\begin{equation}\label{eq:convlow-cone}
\exists k >0 \quad \forall \gamma \in \K, \quad \rho \stackrel{(\K)}{\ast} \rho(\gamma) \geq k \ \rho(\gamma),
\end{equation}
then, for any $\cc \neq 0$,  the law of $\mu(\K)$ under $\Prhoc$ dominates $\poi ^+_{2k}$.
\end{proposition}
\begin{proof}
The proof is very similar to the one of Proposition \ref{prop:comp}, therefore it is omitted.
\end{proof}

\section{L\'evy bridge associated with a diffuse jump measure} \label{sec:3}

Our main interest is now to consider pure jump L\'evy processes and their bridges, see e.g. \cite{PrZamb04,fitzsimmons1995occupation,mansuy2005harnesses,hoyle2011levy} for their construction and their application in various frameworks.
The canonical space is $\Omega:= \mathbb{D}(I;\cX)$, the  c\`adl\`ag paths defined on $I:=[0,1]$ with values in $\cX$.
So, to rely with the above formalismus we associate canonically to any path $Z \in \Omega$ the (jump)  point measure on $\cXt : = I \times \cX$ given by
$$
\mu_Z:= \sum_{t : \Delta Z_t \neq 0} \delta_{(t, \Delta Z_t)} \in \CSt.
$$
For the sake of simplicity we only state our results for one-dimensional processes ($d=1$). Nevertheless they hold also for multidimensional processes because we require the L\'evy measure to have a density with respect to the Lebesgue measure.\\

We suppose in the whole section that the jump measure is finite, which means that we are dealing with compound Poisson processes. The generalization to an infinite jump measure is postponed to the Remark \ref{rem:JumpMeasInfinite}.

\subsection{Characterization of  L\'evy bridges }

We first define how to split a canonical path $Z \in \Omega$ in replacing one of its jumps, say $\Delta Z_{t}$, by two other jumps at other times.
\begin{definition}[Path jump splitting] \label{def:Theta}
Let $Z\in \Omega$ be a path and let $\gamma=(t , \Delta Z_{t})\in \cXt$ be a jump time and a jump size of $Z$ or with other words, an atom of $ \mu_Z$. For $\gamma_1=(s_1,x_1) \in \cXt$ and $\gamma_2=(s_2,x_2) \in \cXt$ we define the splitting map $\Theta_{\gamma,\gamma_1,\gamma_2}$ on paths   as follows:
\be
\Theta_{\gamma,\gamma_1,\gamma_2} Z =  Z- \Delta Z_{t} \IND_{[t,1]} + x_{1} \IND_{[s_1,1]} +
x_2 \IND_{[s_2,1]} .
\ee
\end{definition}
This transformation corresponds at the level of point measures on $\cXt$ to the splitting of an atom $\gamma \in \mu_Z$ into the two atoms $\gamma_{1},\gamma_{2}$. \\
More precisely, we are interested in transformations such that the resulting global jump size of $Z$ stays unchanged. So the new jump sizes, $x_1$ and $x_2$, have to satisfy $x_1 + x_2 =\Delta Z_{t}$. Moreover, choosing times and sizes of the new jumps uniformly at random, we define the following operator.

\begin{definition}[Uniform jump split ] \label{def:ADP}
The operator $\ADP$, acting on non negative functionals $F$ on $\cXt^{3}\times\Omega$, is defined by:
$$\begin{aligned}
\ADP &F \, (Z )= \sum_{t: \Delta Z_t \neq 0} \int_{\tilde{\Gamma}^2}
F((t, \Delta Z_t),\gamma_{1},\gamma_2,\Theta_{(t, \Delta Z_t),\gamma_1,\gamma_2} Z) \,
dx_{1} \delta_{\Delta Z_{t}-x_{1}}(dx_2)  ds_{1} ds_{2}
\\&\text{where}\quad \gamma_{1}:=(s_{1},x_{1}) \text{ and }\gamma_{2}:=(s_{2},x_2).
\end{aligned}$$
\end{definition}
This transformation cancels, one after the other, each jump of the path $Z$ and replace it by two jumps whose sizes add up to the size of the removed jump.

\begin{proposition}
Let $\Pnu$ be the pure jump L\'evy process with L\'evy  measure $\nu(dx)$ supposed to be finite and diffuse with positive density function $\nu(x)$. Let $\Enu$ denote the expectation under $\Pnu$. Then we have for any non negative test functional $F $ on $\cXt^{3}\times\Omega$,
\begin{multline} \label{e3}
\Enu \left[ \ADP F\right] =\\
 \Enu \bigg[ \sum_{\substack{s_{1}\neq s_{2}: \Delta Z_{s_{1}}\not = 0\\ \Delta Z_{s_{2}}\not = 0}}
\big[ \int_{I}
F\big((t, \Delta Z_{s_{1}}+\Delta Z_{s_{2}}),(s_{1},\Delta Z_{s_{1}}),(s_{2},\Delta Z_{s_{2}}),Z\big)dt\big]  \, \chi_{\nu}(\Delta Z_{s_{1}},\Delta Z_{s_{2}})\bigg]
\end{multline}
where the function $\chi_\nu$ is defined as in \eqref{eq:loopcharact} by
$
\chi_{\nu}(x_{1},x_{2}):= \displaystyle \frac{\nu(x_{1}+x_{2})}{\nu(x_{1}) \nu( x_{2} )}.
$
\end{proposition}

\begin{proof}
We recognize the expectation of the random sum in the right hand-side as the integral with respect to the second-order factorial Campbell measure $\Camp_{\Pnu}^{(2)}$
and follow the same way as in the proof of Proposition \ref{prop:loopcharact}. Starting with the left hand-side, after integrating under $\delta_{\Delta Z_{t}-x_{1}}(dx_2) $, one gets
\begin{multline}
\Enu \left[ \ADP F\right] =\nonumber\\
 \int_{I\times\cXt^2 \times \mathcal{M}(\tilde{\Gamma})}
F\Big((t,\Delta Z_{t}),(s_{1},x_{1}),(s_{2}, \Delta Z_{t}-x_{1}),
 Z- \Delta Z_{t} \IND_{[t,1]} + x_{1} \IND_{[s_1,1]} +
(\Delta Z_{t}-x_{1}) \IND_{[s_2,1]}\Big)\,\nonumber\\
 dx_{1} ds_{1} ds_{2}\, \mu_Z (d\gamma) \Pnu(d\mu_{Z}).\nonumber
\end{multline}
By Mecke's formula we can rewrite the integral under the intensity measure of $\mu_{Z}(d\gamma),$ that is under $\nu(y)dydt$
\begin{multline}
 \int_{\mathcal{M}(\tilde{\Gamma})}\int_{\cXt} \int_{I\times\cXt} F\big((t,y),(s_{1},x_{1}),(s_{2}, y-x_{1}), Z + x_{1} \IND_{[s_1,1]} +
(y-x_{1}) \IND_{[s_2,1]}\big) \nonumber\\
dx_{1} \,ds_{1} \,ds_{2}\, \nu(y)\,dy\,dt\, \Pnu(d\mu_{Z}).\nonumber
\end{multline}
Now we change the order of integration so that we first integrate in $y$ and then we change the variable setting  $y=x_1 +x_2.$ This results to the following expression
\begin{multline}
 \int_{\mathcal{M}(\tilde{\Gamma})} \int_{I\times\cXt} \int_{\cXt} F\big((t,x_1 + x_2),(s_{1},x_{1}),(s_{2}, x_2 ), Z + x_{1} \IND_{[s_1,1]} +
x_{2} \IND_{[s_2,1]}\big)\nonumber\\
\nu( x_1 + x_2 )dx_{2}\, dt\, dx_{1}\, ds_{1}\, ds_{2}\, \Pnu(d\mu_{Z}).\nonumber
\end{multline}
Further we divide and multiply by $\nu(x_{1})\nu(x_{2})$ and recognise the terms which correspond to the function $\chi_{\nu},$ and also the intensity measures $\nu(x_{i})\, dx_{i}\, ds_{i}$, $i=1,2,$ which are involved in the bivariate Mecke formula (\ref{eq:Meckebivariate}). We apply the latter and obtain
\begin{eqnarray}
&& \int_{\mathcal{M}(\tilde{\Gamma})} \int_{\cXt^{2}} \int_{I} F\big((t,x_1 + x_2),(s_{1},x_{1}),(s_{2}, x_2 ), Z \big)dt\,
\chi_{\nu}( x_1 , x_2 ) \Camp_{\Pnu}^{(2)} (d\gamma_1,d\gamma_2,d\mu).\nonumber
\end{eqnarray}
By the definition of the second order Campbell measure, this rewrites to the expression on the right hand-side of \eqref{e3}, which ends the proof.
\end{proof}

As in Corollary \ref{cor:densite} we can reformulate this result as the absolute continuity with respect to $\Pnu$ of the image measure of $\Pnu$ under the splitting operator $\ADP$. Choosing test functions in \eqref{e3} of the form
$$
F(\gamma,\gamma_{1},\gamma_{2},Z):=
\varphi(x_1) \,
\frac{\IND_{\#\{t : \Delta Z_t \neq 0\}>1}}{\#\{t : \Delta Z_t \neq 0\} - 1} \, \tilde F (Z)
$$
where  $\varphi$ is a probability density on $\Gamma$,
and applying $\ADP$ we indeed
split one randomly chosen jump of any path (having at least two jumps) into two new jumps, the random size of the first one following a law with density $\varphi$. Thus, one obtains the following result.
\begin{corollary}\label{cor:dencite_levy}
For any test function $\tilde F $ for which $\ADP \tilde F$ is $\Pnu$-integrable,
\be\label{e4}
\Enu\Big[\ADP \tilde F(Z)\Big]=
\Enu \Big[ \tilde F (Z)\, {\bf D}_\nu(Z) \Big]
\ee
where
\be
{\bf D}_\nu(Z) := \frac{\IND_{\#\{t : \Delta Z_t \neq 0\}>1}}{\#\{t : \Delta Z_t \neq 0\}-1}
 \sum_{\substack{s_{1}\neq s_{2}: \Delta Z_{s_{1}}\not = 0\\ \Delta Z_{s_{2}}\not = 0}} \varphi(\Delta Z_{s_{1}}) \chi_{\nu}(\Delta Z_{s_{1}},\Delta Z_{s_{2}}).
\ee
\end{corollary}
\vspace{3mm}
Following the agenda of the previous section we revisit the identity \eqref{e3} and prove that indeed it characterizes bridges of the pure jump L\'evy process $\Pnu$.  \\
Consider the family $(\Pnuxy ,x,y\in\cX) $ of bridges of the L\'evy process $\Pnu$ between time $0$ and time $1$.  They can be constructed as a regular version of the family of  conditional laws $\Pnu (\cdot \,|Z_{0}=x, Z_{1}=y),\ x,y\in\cX,$ see \cite[Proposition 3.1]{PrZamb04}.
We then obtain the following result.
\begin{theorem}\label{thm:reciproc_of_levy_fin}
The identity \eqref{e3} remains valid under any bridge $\Pnuxy $ of the pure jump L\'evy process $\Pnu$.
Reciprocally, consider a pure jump  process $Q$ pinned at time 0 and 1 to two values $x \not = y$, that is $Q(Z_0=x)=Q(Z_1=y)=1$.  If the following identity holds
\begin{multline} \label{e7}
E_{Q} \left[ \ADP \Phi\right] =\\
 E_{Q} \bigg[\sum_{\substack{s_{1}\neq s_{2}: \Delta Z_{s_{1}}\not = 0\\ \Delta Z_{s_{2}}\not = 0}}
\int_{I}
\Phi((t, \Delta Z_{s_{1}}+\Delta Z_{s_{2}}),(s_{1},\Delta Z_{s_{1}}),(s_{2},\Delta Z_{s_{2}}),Z)dt\, \chi_{\nu}(\Delta Z_{s_{1}},\Delta Z_{s_{2}})\bigg]
\end{multline}
 then $Q$ coincides with the bridge $\Pnuxy$.\\
 If the pure jump  process $Q$ is pinned at time 0 and 1 to the same value $x$ (that is it carries only loops which start and end in $x$) and satisfies Identity \eqref{e7} then
 $$
Q( \, \cdot \,|\#\{t : \Delta Z_t \neq 0\}\geq 1) = \Pnuxx ( \, \cdot \,|\#\{t : \Delta Z_t \neq 0\}\geq 1).
 $$
 With other words, $Q$ and the bridge $\Pnuxx$ coincide on the set of non constant paths.
\end{theorem}
\begin{proof}
To show that bridges of the process $\Pnu$ satisfy formula (\ref{e3}) is straightforward by disintegration of $\Pnu$ as mixture of its bridges.\\
To show that reciprocally, any pinned pure jump process $Q$ which satisfies \eqref{e7} coincides with a bridge of $\Pnu$, we exploit the following duality between bridges of pure jump processes and point processes pinned by their first moment: the point measure $\mu_{Z}$ has a fixed first moment $\B(\mu_{Z})=\cc$ if and only if the corresponding pure jump process has fixed initial and  final values $x$ and $y$ satisfying $x-y=\cc$. This together with Theorem \ref{th:main} leads to the conclusion.
\end{proof}

Since Identity \eqref{e7} is linear as a function of $Q$, and since the integrated bivariate function $\chi_{\nu}$ does not depend on the boundary conditions $x$ and $y$, \eqref{e7} eventually characterises the set of all mixtures of bridges $(\Pnuxy ,x,y\in\cX) $, called in the literature the {\em reciprocal class} associated with $\Pnu$, see e.g. \cite{LRZ}.\\

\vspace{3mm}

Approximating L\'evy processes whose L\'evy measure has an infinite mass by a sequence of compound Poisson processes, one obtains the following generalization of the previous theorem.
\begin{remark} \label{rem:JumpMeasInfinite}
Still if its diffuse jump measure is {\em infinite},
a pure jump  process $Q$  is in the reciprocal class of $\Pnu$ if and only if the identity  (\ref{e7}) holds for all continuous bounded test functions $\Phi$ on $\cXt^{3}\times\Omega$ as soon as $\chi_{\nu}$ is $\Camp_{Q}^{(2)}$-integrable. Indeed consider, for any $n \in \N$, the compound Poisson approximation $Z^{n}$ obtained from the initial L\'evy process $Z$ by canceling its jumps whose size is  smaller than $\frac{1}{n}:
Z^{n}_{t}:=Z_{t}\cdot \IND_{|\Delta Z_{t}|>\frac{1}{n}}.
$
Its L\'evy measure is now finite, given by
$
\nu^{n}(dx)=\nu^{n}(x) \, dx:=\nu(x)\IND_{|x|>\frac{1}{n}}\, dx.
$
and Identity (\ref{e3}) holds under $\Pnun$. Applying it  to the
cut-off functions
$
F^{n}(\gamma, \gamma_{1}, \gamma_{2}, Z):= \Phi(\gamma, \gamma_{1}, \gamma_{2}, Z) \IND_{|\Delta Z_{t}|>\frac{1}{n}}
$
where $\Phi$ is any continuous bounded test function, leads to an identity which converges towards  (\ref{e7}) when $n$ grows.
\end{remark}

\subsection{Sampling a L\'evy bridge}\label{subs:simulation}

In this subsection we describe heuristically how to construct a sampler for a L\'evy bridge. Indeed, the basic idea is to construct a dynamic on the pure jump path space whose stationary measure would be the law of  a L\'evy bridge. This generalizes to jump processes some of the results presented in \cite{hairer2005analysis,hairer2007analysis} for diffusion processes. \\

Consider a functional $\Phi$ of the form
\[
\Phi((t,\Delta Z_t),\gamma_1,\gamma_2,Z ) = \left[F(Z) - F(Z + \Delta Z_t \mathbf{1}_{[t,1]} - x_1 \mathbf{1}_{[s_1,1]} - x_2 \mathbf{1}_{[s_2,1]})\right]
\varphi(x_1), \]
where the test functional $F$ is bounded measurable, the density function $\varphi$ is rapidly decaying at infinity and as before, $\gamma_{1}:=(s_{1},x_{1}) $ and $\gamma_{2}:=(s_{2},x_2)$.
 Equation \eqref{e7}, satisfied by any bridge of $\Pnu$,  rewrites for such $\Phi$ as
\begin{eqnarray*}
&&\mathbb{E}_{\Pnuxy} \Bigg[ \sum_{t: \Delta Z_t \neq 0} \int_{I^2 \times \Gamma}
\left[F(Z-\Delta Z_t \mathbf{1}_{[t,1]} + x_1 \mathbf{1}_{[s_1,1]} + x_2 \mathbf{1}_{[s_2,1]}) )-F(Z)\right] \varphi(x_1) d x_1 d s_1 d s_2 \\
&&+  \sum_{\substack{s_{1}\neq s_{2}: \\ \Delta Z_{s_{1}}\Delta Z_{s_{2}}\not = 0}}
\varphi(\Delta Z_{s_{1}})\chi_{\nu}(\Delta Z_{s_{1}},\Delta Z_{s_{2}})\\
&& \quad \quad  \int_{I} \big[F(Z +(x_1+x_2)\mathbf{1}_{[t,1]}- x_1 \mathbf{1}_{[s_1,1]}-x_2 \mathbf{1}_{[s_2,1]}) -F(Z)\big] dt \Bigg]=0.
\end{eqnarray*}
This identity suggests that the bridges $\Pnuxy$ can be interpreted as the invariant law of a Markov process on the path space regulated
by two mechanism: either jumps  split/fragmentate (first term) or jumps coalesce (second term). More precisely, if $Z$ is the current state of the process, then
\begin{itemize}
\item Each jump $(t,Z_t)$ of the path $Z$ splits at rate $1$; when this happens, the jump at $t$ is removed, and is replaced by two new jumps $(s_1,x_1)$ and $(s_2,x_2)$ which are sampled according to the following rules.
\begin{itemize}
\item The jump times $s_1,s_2$ are chosen uniformly at random in $[0,1]^2$.
\item The first jump size $x_1$ is sampled from the probability law with density $\varphi$ and the second jump size is set to be $x_2:= \Delta Z_t - x_1 $

\end{itemize}
\item Each ordered pair $(s_1,\Delta Z_{s_1}),(s_2,\Delta Z_{s_2})$ of jumps of $Z$ coalesce at rate $\varphi(Z_{s_1})\chi_{\nu}(Z_{s_1},Z_{s_2})$; when this happens, the two jumps are removed from  $Z$ and replaced by a single jump $(t,\Delta Z_t)$ sampled according to the following rules.
\begin{itemize}
\item The jump time $t$ is sampled uniformly at random in $[0,1]$.
\item The jump size is the sum of the sizes of the removed jumps: \\
\mbox{$\Delta Z_t := \Delta Z_{s_1}+ \Delta Z_{s_2}$}
\end{itemize}
\end{itemize}

\subsection{Stochastic comparison between L\'evy bridges} \label{sec:compLevyBridges}

In this section we apply the above results to investigate domination properties for bridges of pure jump L\'evy processes.\\
We consider L\'evy measures having the form $\nu(dx)=\lambda \  f(x)dx$ where the constant $\lambda>0$ encodes the intensity of the jumps per unit interval, and the function $f$ on $\R$ is a probability density encoding the distribution of the jumps. It then corresponds to the assumption made in the beginning of Section \ref{sec:2.2}.
Thus, supposing the density $f$ to be positive, domination conditions (\ref{eq:cconvup}), respectively (\ref{eq:cconvlow}), rewrite:
\begin{equation} \label{inequ:conv}
 \exists\,  K < \infty, \,
 \sup_{x \in \R} \frac{f * f(x)}{f(x)}\leq K , \quad
 \textrm{ resp. }
 \quad \exists \, k >0, \,
 \inf_{x \in \R} \frac{f* f(x)}{f(x)}\geq k .
\end{equation}
Our aim is to compare the law of the number of jumps of a L\'evy bridge with a Poisson distribution.
We consider two specific families of L\'evy bridges:
their L\'evy measures  are of Cauchy-type with densities $f_\alpha$ or of symmetric exponential-type with densities $g_\beta$, where :
$$
f_{\alpha}(y):= \frac{r_{\hspace{-0.3mm}\alpha}}{1+|y|^{\alpha}}, \quad \alpha >1, \quad \text{and}\quad
g_{\beta}(y):= r_{\hspace{-0.3mm}\beta}e^{-|y|^{\beta}},\quad \beta>0.
$$
Here $r_{\hspace{-0.3mm}\alpha}>0, r_{\hspace{-0.3mm}\beta} >0$ denote the normalising constants. \\

\textbf{Stochastic comparison for the Cauchy-type family}.\\
We now prove that both inequalities in \eqref{inequ:conv} are satisfied by this family of jump laws or equivalently, we prove that the function $H_{\alpha}(x):= \displaystyle \frac{f_\alpha * f_\alpha}{f_\alpha}(x)$ is uniformly bounded from above and from below (by a positive constant). First notice the integral representation:
$
H_{\alpha}(x)= \displaystyle \int_{\R}h_{\alpha}(x,y)\,dy,
$
with
$$
h_{\alpha}(x,y):= \displaystyle \frac{1+|x|^{\alpha}}{(1+|y+\frac{x}{2}|^{\alpha})(1+|y-\frac{x}{2}|^{\alpha})}.
$$
Since the function $h_{\alpha}$ is symmetric in $x$ and $y$, it is enough to consider $h_{\alpha}(x,y)$ for $x>0, y>0.$ \\
{\it Upper bound.} Since $H_\alpha$ is continuous it is bounded from above on the interval $[0,1]$.
 So let us consider $x\in [1,+\infty[.$
We decompose $H_\alpha(x)$ into two integrals:
$$
\frac{1}{2} H_{\alpha}(x) = \int_{0}^{x/2} h_{\alpha}(x,y)\,dy + \int_{x/2}^{+\infty} h_{\alpha}(x,y)\,dy.
$$
Now
$$
\int_{0}^{x/2} h_{\alpha}(x,y)\,dy\leq \frac{1+x^{\alpha}}{1+(x/2)^{\alpha}}\int_{0}^{x/2}\frac{dz}{1+z^{\alpha}}
< \frac{2^{\alpha}(1+x^{\alpha})}{2^{\alpha}+x^{\alpha}}\int_{0}^{+\infty}\frac{dz}{1+z^{\alpha}}
$$
which is uniformly bounded for $x \in [1,+\infty[$ since $\alpha$ is supposed to be larger than $1$.
Similarly
\begin{eqnarray*}
\int_{x/2}^{+\infty} h_{\alpha}(x,y)\,dy
&\leq& \frac{1+x^{\alpha}}{1+(x/2)^{\alpha}}\int_{x/2}^{+\infty} \frac{1}{1+(y-x/2)^{\alpha}}dy\\
&=& \frac{2^{\alpha}(1+x^{\alpha})}{2^{\alpha}+x^{\alpha}}\int_{0}^{+\infty}\frac{dz}{1+z^{\alpha}}
\end{eqnarray*}
which is uniformly bounded for $x \in [1,+\infty]$.\\
{\it Lower bound.} As before, it is enough to consider $x\in [1,+\infty[.$
$$
\frac{1}{2} H_{\alpha}(x) \geq
\int_{0}^{x/2} h_{\alpha}(x,y)\,dy \geq \frac{1+x^{\alpha}}{1+x^{\alpha}}\int_{0}^{x/2}\frac{dz}{1+z^{\alpha}}
\geq \int_{0}^{1/2}\frac{dz}{1+z^{\alpha}} >0
$$
Hereby we have shown that $\nu_\alpha= \lambda \  f_\alpha(x)dx$ satisfies  both inequalities \eqref{eq:cconvup} and \eqref{eq:cconvlow} for some constants $K_\alpha$ and $k_\alpha.$
 Due to Proposition \ref{prop:comp} we conclude that the distribution of the number of jumps for any bridge of a  L\'evy process with Cauchy-type jump distribution, conditioned to have at least one jump, is stochastically equivalent with a Poisson law conditioned to stay positive.\\
For $\alpha=2$ the law of the jumps is a Cauchy distribution with density $\displaystyle f_{2}(y)=\frac{1}{\pi(1+y^{2})}$. Thus $f_2*f_2(y)=\frac{2}{\pi(4+y^{2})}$ and we obtain the explicit bounds: $\frac{1}{2} \leq H_{\alpha}(x) \leq 2$. Therefore, as application of  Proposition \ref{prop:comp} and  Corollary \ref{cor:comp_exp}, the following holds.
\begin{proposition}
The distribution of the number of jumps for any bridge of a L\'evy process with Cauchy jump distribution, supposing it is larger than 0, is stochastically dominated by (resp. dominates) a Poisson law with parameter $4\lambda$ (resp. $\lambda$) conditioned to stay positive.
 Therefore its expected number belongs to  $[\frac{\lambda}{1-e^{-\lambda}},\frac{4\lambda}{1-e^{-4\lambda}}].$ For $\lambda = 1$, this interval is equal to $[1.58;4.07]$.
\end{proposition}
Notice once more that these comparisons do not depend on the height of the bridge, as soon as it differs from 0. \\

\textbf{Stochastic comparison for the symmetric exponential-type family}.
We now prove that (only) the second inequality in \eqref{inequ:conv} is satisfied by the family of jump densities $g_\beta$ or equivalently, we prove that the function $G_{\beta}(x):= \displaystyle \frac{g_\beta * g_\beta}{g_\beta}(x)$ is uniformly bounded from below by a positive constant. First notice the integral representation:
\begin{eqnarray*}
G_{\beta}(x) =  e^{|x|^{\beta}}\int_{\R}\tilde{\gbeta}(x,y) \, dy,
\end{eqnarray*}
 where $\tilde{\gbeta}(x,y):=e^{-|y+x/2|^{\beta}}e^{-|y-x/2|^{\beta}}$.
Remark that the function $\tilde{\gbeta}(x,y)$ is even in $y$ and symmetric in $x.$\\
{\em First case}: $0<\beta<1.$ The graph of $y \mapsto \tilde{\gbeta}(x,y)$ is bimodal for $x\neq 0$ and becomes unimodal for $x=0.$ The value  of $G_{\beta}$ at $x=0$ is $G_{\beta}(0) = \frac{2}{\beta 2^{1/\beta}}\Gamma(1/\beta).$ On the compact interval $[-1/2, 1/2]$,  the continuous map  $G_{\beta}$ is bounded from below by a positive constant.
For $|x|>1/2$ one can inscribe under the graph of $y \mapsto \tilde{\gbeta}(x,y)$ equal triangles with vertices at $A_{+}:=(x/2, e^{-|x|^{\beta}})$ and $A_{-}:=(-x/2, e^{-|x|^{\beta}})$, having as  sides the tangents at each of the vertices $A_{+}, A_{-}$ and with height $h=e^{-|x|^{\beta}}.$ Then
$$
\inf_{|x|>1/2}G_{\beta}(x)\geq \inf_{|x|>1/2} e^{|x|^{\beta}}2e^{-|x|^{\beta}}(\frac{|x|^{1-\beta}}{\beta}+\frac{|x|}{2})\geq \frac{2^{\beta}}{\beta }+\frac{1}{2}.
$$
 This shows that the function $G_{\beta}$ is uniformly bounded from below by a positive constant. \\
{\em Second case}: $\beta\geq 1.$
The graph of  $y \mapsto \tilde{\gbeta}(x,y)$  becomes unimodal, and since the function is symmetric we consider only the case $x>0.$ The unique maximum of this function is at the point $x/2.$ Analysing the integrals over $[0,x/2]$ and $[x/2,+\infty)$ for $0\leq x \leq1$ and $x>1$ and using the respective asymptotical behaviour of the incomplete Gamma function, we get that $G_{\beta}(x)$ is again uniformly bounded from below by a positive constant $k_\beta$.
Due to Proposition \ref{prop:comp} we conclude that the distribution of the number of jumps for any bridge of a  L\'evy process with symmetric exponential-type jump distribution, conditioned to have at least one jump, stochastically dominates a Poisson law with parameter $2\lambda \, k_\beta $ conditioned to stay positive.\\
For $\beta=1$ the law of the jumps is a Laplace distribution with density
\mbox{$g_{1}(y)=e^{-|y|}/2$}. We compute explicitly $G_1(x) = \frac{1}{2} (1 + |x|)$ and obtain as lower bound  $k_1= 1/2  $. \\
For  $\beta=2$ the law of the jumps is the standard normal distribution with density $g_{2}(y)=e^{-y^{2}}/\sqrt{\pi}$. We compute explicitly $G_2(x) = \frac{1}{\sqrt{2}} e^{x^2}$ and obtain as lower bound $k_2= 1/\sqrt{2}$.\\
Thus,  as application of  Proposition \ref{prop:comp} and  Corollary \ref{cor:comp_exp}, the following holds.
\begin{proposition}
The distribution of the number of jumps
for any bridge of a L\'evy process with
 Laplace (resp. standard Gaussian) jump distribution, conditioned to have at least one jump,
stochastically dominates  a Poisson law with parameter $\lambda$ (resp. $ \sqrt{2} \lambda$) conditioned to stay positive.
 Therefore its expected number is not less  than $\lambda/(1-e^{-\lambda})$ (resp. $ \sqrt{2} \lambda(1-e^{- \sqrt{2} \lambda})$). For $\lambda = 1$ these bounds are equal to $1.58$  (resp. $1.07$).
\end{proposition}
Once more, it is remarkable that these bounds do not dependent of the height of the bridge.

\section{Periodic Ornstein-Uhlenbeck processes}
\label{sec:4}

We now generalize some result of the previous section to the case of linear diffusion driven by a L\'evy process. Introducing a damping force in the random dynamics, we
consider the real-valued Langevin equation with damping parameter $c\in \R^*$
\begin{equation} \label{OUwithLevy}
dX_t = - c X_t \, dt + dZ_t, \quad t \in [0,1]
\end{equation}
where $Z$ is the L\'evy process with law $\Pnu$.  The measure $\nu$ is as before a diffuse finite L\'evy measure on $\R^*$.
Suppose moreover that the solution of this SDE is periodized, that is satisfies the boundary conditions $X_0=X_1$. This process, studied in \cite{PED02,PEDSATO}, is called periodic Ornstein-Uhlenbeck with parameter $c$ and background driving L\'evy process $Z$, short PerOU-L\'evy process. We denote its law  by $\PerOU$. \\
Notice that, if one replaces in \eqref{OUwithLevy} the pure jump process $Z$  by a Brownian motion, one recovers the known  periodic Ornstein-Uhlenbeck process, whose properties as a mixture of Ornstein-Uhlenbeck bridges  are discussed in \cite[Theorem 5.1]{RT02}.   A review of its semi-martingale properties can be found in \cite{roelly2016convoluted}.\\

 Indeed the periodic solution of  \eqref{OUwithLevy} is the image measure of $\Pnu$ under  the map $\mathbb{X}^c: \mathbb{D}([0,1];\R) \rightarrow \mathbb{D}([0,1];\R)$ given by:
\begin{eqnarray*}
\mathbb{X}^c_0(Z) &=& \mathbb{X}^c_1(Z) = \frac{1}{e^c-1} \int_{0}^1 e^{c s} \, d Z_s, \\
\mathbb{X}^c_t(Z)  &=& e^{-ct} \, \mathbb{X}^c_0(Z) + e^{-ct} \int_{0}^{t} e^{cs}\, d Z_s.
\end{eqnarray*}

We would like to exhibit an identity generalizing \eqref{e7} satisfied by the PerOU-L\'evy process.
To this aim we generalize the former operators given  in Definitions \ref{def:Theta} and \ref{def:ADP} and introduce new time-weighted operators $\Theta^c_{\gamma,\gamma_1,\gamma_2} $ and $\ADPa$ which take into account the $c$-damping of the paths. They are defined by composing $\Theta_{\gamma,\gamma_1,\gamma_2} $ and $\ADP$ with  the map $\mathbb{X}^c$. More precisely, for any $\gamma=(t , \Delta Z_{t})\in [0,1]\times \R^*$ jump time and jump size of $Z$, any $\gamma_1=(s_1,x_1) \in [0,1]\times \R^*$ and $\gamma_2=(s_2,x_2) \in [0,1]\times \R^*$, we introduce the {\em time-weighted path jump splitting} by
\bes
\Theta^c_{\gamma,\gamma_1,\gamma_2}(Z) = Z + \mathbb{X}^c(  - \Delta Z_{t} \IND_{[t,1]} + x_1 \IND_{[s_1,1]} + x_2 \IND_{[s_2,1]} ).
\ees
Randomizing time and size of the new jumps, one define the following operator on positive test functions $F$ defined on $([0,1]\times \R^*)^{3}\times\Omega$:
\bes
\ADPa F \, (Z ):= \int_{([0,1]\times \R^*)^{3}}  F(\gamma,\gamma_{1},\gamma_{2},\Theta^c_{\gamma,\gamma_1,\gamma_2} Z) \,
dx_{1}\delta_{\Delta Z_{t}-x_{1}}(dx_2) ds_{1} ds_{2}\, \mu_Z (d\gamma).
\ees

\begin{proposition}\label{prop:OUperiodic}
Let $\PerOU$ be the periodic Ornstein-Uhlenbeck process with parameter $c$ driven by the L\'evy process $Z$ with L\'evy measure $\nu$. It satisfies the following identity for any positive test functions
$F $:
\begin{eqnarray}\label{eOU2}
&&E_{\PerOU} \left[ \ADPa F\right] = \\
&& E_{\PerOU} \bigg[\sum_{\substack{s_{1}\neq s_{2}:\Delta Z_{s_{1}}\not = 0 \\ \Delta Z_{s_{2}}\not = 0}}
\chi_{\nu}(\Delta Z_{s_{1}},\Delta Z_{s_{2}})
\int_{0}^1
F\big((t, \Delta Z_{s_{1}}+\Delta Z_{s_{2}}),(s_{1},\Delta Z_{s_{1}}),(s_{2},\Delta Z_{s_{2}}),Z\big)dt \bigg] \nonumber
\end{eqnarray}
where $\chi_{\nu}$ is the reciprocal characteristic associated with the measure $\nu$, given as before by \eqref{eq:loopcharact}.
\end{proposition}
\begin{proof}
By using the linearity of the map $\mathbb{X}^c$
\bes
E_{\PerOU} \left[ \ADPa F\right] = \Enu \left[ \ADP (F \circ  \mathbb{X}^c)  \right] .
\ees
Applying identity \eqref{e3} under $\Pnu$,
\begin{eqnarray*}
&&
\Enu \left[ \ADP (F \circ  \mathbb{X}^c)  \right] =\\
&& \Enu \bigg[\sum_{\substack{s_{1}\neq s_{2}:\Delta Z_{s_{1}}\not = 0 \\ \Delta Z_{s_{2}}\not = 0}}
\int_{I}
F((t, \Delta Z_{s_{1}}+\Delta Z_{s_{2}}),(s_{1},\Delta Z_{s_{1}}),(s_{2},\Delta Z_{s_{2}}),\mathbb{X}^c(Z))dt\, \chi_{\nu}(\Delta Z_{s_{1}},\Delta Z_{s_{2}})\bigg]
\end{eqnarray*}
Since $\mu_{\mathbb{X}^c} = \mu_{Z}$ we can rewrite the right hand side of the previous
equation as the right hand side of \eqref{eOU2}.
\end{proof}
\bibliographystyle{plain}
\bibliography{Ref1}

\begin{thebibliography}{10}

\bibitem{CDPR}
{G.} Conforti, {P.} Dai~Pra, and {S.} Roelly.
\newblock Reciprocal classes of jump processes.
\newblock {\em Journal of Theoretical probability}, (2015).

\bibitem{CL15}
{G.} Conforti and {C.} L{\'e}onard.
\newblock Reciprocal classes of random walks on graphs.
\newblock {\em Stochastic Processes and their Applications}, 127(6):1870--1896,
  2017.

\bibitem{CR}
{G.} Conforti and {S.} Roelly.
\newblock Bridge mixtures of random walks on an abelian group.
\newblock {\em Bernoulli}, 23(3):1518--1537, 2017.

\bibitem{daley2007introduction}
{D.J.} Daley and {D.} Vere-Jones.
\newblock {\em An introduction to the theory of point processes: volume II:
  general theory and structure}.
\newblock Springer Science \& Business Media, 2007.

\bibitem{fitzsimmons1995occupation}
{P.J.} Fitzsimmons and {R.K.} Getoor.
\newblock Occupation time distributions for {L}{\'e}vy bridges and excursions.
\newblock {\em Stochastic processes and their applications}, 58(1):73--89,
  1995.

\bibitem{hairer2007analysis}
{M.} Hairer, {A.M.} Stuart, and {J.} Voss.
\newblock Analysis of spdes arising in path sampling part ii: The nonlinear
  case.
\newblock {\em The Annals of Applied Probability}, pages 1657--1706, 2007.

\bibitem{hairer2005analysis}
{M.} Hairer, {A.M.} Stuart, {J.} Voss, and {P.} Wiberg.
\newblock Analysis of spdes arising in path sampling. part i: The gaussian
  case.
\newblock {\em Communications in Mathematical Sciences}, 3(4):587--603, 2005.

\bibitem{hoyle2011levy}
{E.} Hoyle, {L.P.} Hughston, and {A.} Macrina.
\newblock L{\'e}vy random bridges and the modelling of financial information.
\newblock {\em Stochastic Processes and their Applications}, 121(4):856--884,
  2011.

\bibitem{lastpenrose}
{G.} Last and {M.} Penrose.
\newblock {\em Lectures on the Poisson Process}.
\newblock IMS Textbooks. Cambridge University Press, 2018.

\bibitem{LRZ}
{C.} L{\'e}onard, {S.} R{\oe}lly, and {J.C.} Zambrini.
\newblock Reciprocal processes. {A} measure-theoretical point of view.
\newblock {\em Probability Surveys}, 11:237--269, 2014.

\bibitem{mansuy2005harnesses}
{R.} Mansuy and {M.} Yor.
\newblock Harnesses, {L}{\'e}vy bridges and monsieur {J}ourdain.
\newblock {\em Stochastic processes and their applications}, 115(2):329--338,
  2005.

\bibitem{nehring2016splitting}
{B.} Nehring, {M.} Rafler, and {H.} Zessin.
\newblock Splitting-characterizations of the {P}apangelou process.
\newblock {\em Mathematische Nachrichten}, 289(1):85--96, 2016.

\bibitem{PED02}
{J.} Pedersen.
\newblock Periodic {O}rnstein-{U}hlenbeck processes driven by {L}evy processes.
\newblock {\em Journal of applied probability}, pages 748--763, 2002.

\bibitem{PEDSATO}
{J.} Pedersen and {K.-I.} Sato.
\newblock The class of distributions of periodic {O}rnstein-{U}hlenbeck
  processes driven by {L}{e}vy processes.
\newblock {\em Journal of Theoretical Probability}, 18(1):209--235, 2005.

\bibitem{PrZamb04}
{N.} Privault and {J.C.} Zambrini.
\newblock Markovian bridges and reversible diffusion processes with jumps.
\newblock {\em Annales de l'{I}nstitut {H}enri {P}oincar{\'e} (B),
  {P}robababilit\'es et Statistiques}, 40(5):599--633, 2004.

\bibitem{RT02}
{S.} R{\oe}lly and {M.} Thieullen.
\newblock A characterization of reciprocal processes via an integration by
  parts formula on the path space.
\newblock {\em Probability Theory and Related Fields}, 123(1):97--120, 2002.

\bibitem{roelly2016convoluted}
{S.} Roelly and {P.} Vallois.
\newblock Convoluted brownian motion: a semimartingale approach.
\newblock {\em Theory of Stochastic Processes}, 21(2):58--83, 2016.

\end{thebibliography}
\end{document}